\documentclass [12pt]{amsart}
\usepackage {amssymb}
\usepackage {amsmath}
\newtheorem{theorem}{Theorem}[section]
\newtheorem{proposition}{Proposition}[section]
\newtheorem{lemma}{Lemma}[section]

\begin{document}

\setcounter{page}{1}

\title{ON THE IRREDUCIBLE REPRESENTATIONS OF \\
 SOLUBLE GROUPS OF FINITE RANK}

\author{A.V. Tushev}

\address{Department of Mathematics, Dnepropetrovsk National University,\\
Prospect Gagarina 72, Dnepropetrovsk, 49050, Ukraine\\
tushev@mamber.ams.org}

\maketitle

\begin{abstract}
We obtained some sufficient and necessary conditions of existence of faithful irreducible representations of a soluble group $G$  of finite rank over a field $k$ . It was shown that the existence of such representations strongly depends on construction of the socle of the group $G$ . The situation is especially complicated in the case where the field $k$  is locally finite.
\end{abstract}

\section{Introduction}
     We recall that a group $G$ has finite (Prufer) rank if there is an integer $r$ such that each finitely generated subgroup of $G$ can be generated by $r$   elements;  its  rank $r(G)$ is then the least integer $r$ with this property. A group $G$ is said to be polycyclic if it has a finite series in which each factor is cyclic.

     Let $G$ be a group, the subgroup $Soc(G)$ of $ G$ generated by all its minima normal subgroups is said to be the socle of the group $G$  ( if the group $G$ has no minimal normal subgroups then $Soc(G) = 1$). The subgroup $abSoc(G)$ of the group $G$ generated by all its minima normal abelian subgroups is said to be the abelian socle of the group $G$  ( if the group $G$ has no minimal abelian normal subgroups then $abSoc(G) = 1$) . If the group $G$ is soluble then $abSoc(G) = Soc(G)$.

     A nontrivial normal subgroup $E$ of a group $G$ is said to be essential if $E \cap N \ne 1$ for any nontrivial normal subgroup $N$ of $G$.

     If $B$ is an abelian group of finite rank then the spectrum $Sp(B)$ of the group $B$ is the set of prime numbers $p$ such that the group $B$ has an infinite $p$-section.

It was proved in \cite{7} that if a polycyclic group $G$ has a faithful irreducible representation over a locally finite field $k$ then the group $G$ is finite. However, infinite locally polycyclic groups of finite rank may have faithful irreducible representations over a locally finite field $k$. Moreover, in \cite{8}  we  found necessary and sufficient conditions of existence of faithful irreducible representations of locally polycyclic groups of finite rank over a locally finite field $ k$.

	In \cite{4} Gaschutz discovered a critical role of $Soc(G)$ in the theory of representations of finite groups. We denote by $\mathfrak S_{\mathfrak F} $ a class of all groups $G$ such that $Soc(G)$ is an essential normal subgroup of $G$ and all minimal normal subgroups of $G$ are finite. We should note that the class $\mathfrak S_{\mathfrak F} $ is rather large and contains locally normal groups and torsion groups of finite rank. In theorems 1 and 2 of \cite{9} we proved that a group $G \in \mathfrak S_{\mathfrak F}  $ has an irreducible faithful representation over a field $k$ if and only if $chark \notin \pi (abSoc(G))$ and one of the following equivalent conditions holds:

(i) $abSoc(G)$ has a subgroup $H$ such that $abSoc(G)/H$ is a locally cyclic group and $H$ contains no nontrivial $G$-invariant subgroups;

     (ii) $abSoc(G)$ is a locally cyclic $\mathbb{Z}G$-module, where $G$ acts on $abSoc(G)$ by conjugations.

     We should note that in the case of torsion soluble groups of finite rank such criterions were obtained in theorems 2 and 3 of \cite{8}. Thus, the structure of $abSoc(G)$ is the most important for existence of faithful irreducible representations of locally normal and, in particular, finite groups.

	 In the presented paper we are searching necessary and sufficient conditions of existence of faithful irreducible representations of soluble groups of finite rank over a field $k$.
     In theorem 5.1 is proved that if a soluble group of finite rank $G$ has a faithful irreducible representation over a field $k$ then $Soc(G)$ of the group $G$ is a locally cyclic $\mathbb{Z}G$-module, where the group $G$ acts on $Soc(G)$ by conjugations, and  $chark \notin \pi (Soc(G))$. Theorem 5.2 shows that if the field $k$ is not locally finite then the above condition on $Soc(G)$ is also sufficient for existence of an irreducible faithful representation of $G$ over $k$. By theorem 5.3, in the case where the field $k$ is locally finite the condition on $Soc(G)$ is also a criterion of existence of an irreducible faithful representation of $G$ over $k$ with an additional assumption that $Sp(B) \not\subset \{ chark\} $ for any nontrivial torsion-free abelian  normal subgroup $B$ of  $G$.

     However, the condition $Sp(B) \not\subset \{ chark\} $ for any nontrivial abelian torsion-free normal subgroup of $G$ is not necessary for existence of faithful irreducible representations of soluble groups of finite rank over a locally finite field. It follows from a result by Wherfritz \cite{10} in which a simple $kG$-module $W$ was constructed such that $C_G (W) = 1$ , where $k$ is a field of order $p$, $G = B\lambda \left\langle g \right\rangle $ is a torsion-free soluble group of rank 2 and $Sp(B) = \left\{ p \right\}$. So, in the case of locally finite fields the situation is much more complicated.

     \section{On Direct Sums of some Just Infinite Modules}

     An abelian group is said to be minimax if it has a finite series each of whose factor is either cyclic or quasi-cyclic. It is easy to note that for any abelian minimax group $A$ the set $Sp(A)$ is finite

     Let $R$ be a ring, an $R$-module $A$ is said to be just infinite (or $R$ -ji-module for shortness)  if $A$ is infinite and for any proper submodule $V$ of $A$ the quotient module $A/V$ is finite. We also omit $R$ in the notation if the structure of the ring $R$ is not important.

     We should emphasize that below  we always assume that any $R$-ji-module is torsion-free minimax as an abelian group.

     \begin{lemma}
           Let $A =  \oplus _{i = 1}^n A_i $, where $A_i $ are $ji$-modules. Then:\par
      (i) any nonzero submodule of $A$ contains a nonzero $ji$ -submodule;\par
(ii) the set of all  $ji$-submodules of  $A$ is countable.
\end{lemma}
\begin{proof}
     (i) Let $X$ be a nonzero submodule of $A$. Evidently, the module $A$ has a finite series each of whose quotient is a $ji$-module. The intersection of this series with $X$ give us a finite series of submodules of $X$ each of whose quotient is either zero or a $ji$-module and the assertion follows.

     (ii) The proof is by induction on $n$. Let $n = 1$ then $A$ is a $ji$-module and for any proper submodule $B$ of  $A$ the quotient module $A/B$ is finite. Therefore there is an integer  $n$ such that $An \leqslant B$ and, as the set of integers $n$ is countable  and for each such $n$ there is only finite set of submodules  $B$ of  $A$ such that $An \leqslant B$, the set of all submodules of  $A$ is an union of  countable many of finite subsets. So, the set of all  $ji$-submodules of  $A$ is countable.

     Consider now the general case then $A = C \oplus D$, where $C = A_1 $ and $D =  \oplus _{i = 2}^n A_i $. By the induction hypothesis,  the set $\Re $ of all  $ji$-submodules of  $C$ and $D$ is countable.
     Let $B$ be a $j.i.$-submodule of  $A$, if $C \cap B \ne 0$ then the quotient group $B/(C \cap B)$ is torsion-free because so is the quotient group $A/C$ and, as $B$ is a $ji$-submodule, it implies that $B \leqslant C$. The same arguments show that if $D \cap B \ne 0$ then $B \leqslant D$. So, if $C \cap B \ne 0$ or $D \cap B \ne 0$ then $B \in \Re $.

     Let $\Im $ be the set of all $ji$-submodules of $A$ such that  $C \cap B = 0$ and $D \cap B = 0$. Since the set $\Re $ is countable, it is sufficient to show that so is $\Im $. Thus, we can assume that $B \in \Im $. Then the mappings  $\varphi :B \to \Pr _C B$ and $\phi :B \to \Pr _D B$ given by $\varphi :b \mapsto \Pr _C b$ and $\phi :b \to \Pr _D b$ are module isomorphisms. Since the set $\Re $ is countable, the set of submodules $U \oplus V$ of $A$, where $U \leqslant C$ and $V \leqslant D$ are isomorphic $ji$-submodules, is countable. Then it is easy to note that it is sufficient to show countability of the set $\aleph $ of all $ji$-submodules $B$ of $(W)_1  \oplus (W)_2 $ such that $(W)_1  \cap B = (W)_2  \cap B = 0$ and $W = \Pr _{(W)_1 } B = \Pr _{(W)_2 } B$, where  $W$  is a $ji$-module.
     Let $B \in \aleph $, as $(W)_1  \cap B = (W)_2  \cap B = 0$ and  $W = \Pr _{(W)_1 } B = \Pr _{(W)_2 } B$, for any $w \in (W)_1  = W$ there is a unique $v \in (W)_2  = W$ such that $w + v \in B$. It is easy to note that the mapping $\psi _B :W \to W$ given by $\psi _B :w \mapsto v$ is a group automorphism of $W$. So, we have a mapping $\eta :\aleph  \to AutW$ given by $B \mapsto \psi _B $. If $\psi _B  = \psi _K $ for some $K,B \in \aleph $ then for any $w \in (W)_1  = W$ and $v \in (W)_2  = W$ we have $w + v \in B$ if and only if $w + v \in K$ and it easily implies that $K = B$. So, $\eta $ is an injection. Since $W$ is a torsion-free abelian group of finite rank $n$, it is well known that $AutW \leqslant GL_n (\mathbb{Q})$. Then, as the set $GL_n (\mathbb{Q})$ is countable, so is $AutW$ and, as we have an injection  $\eta :\aleph  \to AutW$, the set $\aleph $ is countable.
\end{proof}
\begin{lemma}
     Let $A_i $ be quasicyclic $p$-groups, where $i = 1,2$, and $A =  \oplus _{i = 1}^2 A_i $. Then the cardinality of the set of all subgroups $B$ of $A$ which defines the quasicyclic quotint group $A/B$ is continuum.
\end{lemma}
\begin{proof}
      As for any quasicyclic subgroup $B$ of $A$ the quotient group $A/B$ is quasicyclic, it is sufficient to show that the cardinality of the set of all quasicyclic subgroups  of $A$ is continuum. Let $\Re $ be the set of all  quasicyclic $p$-subgroups $B$ of $A$ such that $B \cap A_1  = B \cap A_2  = 0$. Then it is sufficient to show that the cardinality of the set $\Re $ is continuum. Let  $B \in \Re $ then any element of  $B$ may be uniquely  written in the form $a = a_1  + a_2 $, where $a_1  = \Pr _{A_1 } (a)$ and $a_2  = \Pr _{A_2 } (a)$. If $\alpha  \in Aut(A_2 )$ then elements $a_1  + \alpha (a_2 )$ form another  subgroup $B_\alpha   \in \Re $. Let  $\alpha ,\beta  \in Aut(A_2 )$ if $B_\alpha   = B_\beta  $ then $\alpha (a_2 ) = \beta (a_2 )$ for all $a \in B$ and it implies that $\alpha  = \beta $. Thus, we have an injection  $\alpha  \mapsto B_\alpha  $ of $Aut(A_2 )$ into $\Re $. It is well known that $Aut(A_2 )$ is isomorphic to the group $U(J_p )$ of units of the ring $J_p $ of  $p$-adic integers . As the cardinality of  $U(J_p )$ is continuum, so is the cardinality of  $\Re $.
\end{proof}
\begin{proposition}
      Let $A =  \oplus _{i = 1}^n A_i $, where $A_i $ are $ji$-modules and let $p$ be a prime number such that $p \in Sp(A_i )$ for all $i = \overline {1,n} $.  Then the group $A$ has a subgroup $B$ such that $A/B$ is a quasicyclic $p$-group  and $B$ does not contain nonzero submodules.
\end{proposition}
\begin{proof}
     The proof is by induction on $n$. Let $n = 1$ then $A$ is a $ji$-module and, as $p \in Sp(A_i )$, the group $A$ has a subgroup $B$ such that $A/B$ is a quasicyclic $p$-subgroup. It follows from the definition of $ji$-modules that $B$ does not contain nonzero submodules.

     Consider now the general case then $A = C \oplus D$, where $C = A_1 $ and $D =  \oplus _{i = 2}^n A_i $. By the induction hypothesis, $C$ contains a subgroup $C_1 $ and $D$ contains a subgroup $D_1 $ which do  not contain nonzero submodules and such that $C/C_1 $ and $D/D_1 $ are quasicyclic $p$-groups. Put $B_1  = C_1  \oplus D_1 $.
     Show that the subgroup $B_1 $ does not contain nonzero submodules. Let $X$ be a submodule of $B_1 $. Then $\Pr _{C_1 } (X)$ is a submodule of  $C_1 $ an hence, as $C_1 $ does  not contain nonzero submodules, we see that $\Pr _{C_1 } (X) = 0$. The same arguments show that $\Pr _{D_1 } (X) = 0$ and hence $X = 0$. Thus, $B_1 $ does not contain nonzero submodules.

     Evidently, $A/B_1  = P_1  \oplus P_2 $, where $P_1  = C/C_1 $ and $P_2  = D/D_1 $ are quasicyclic groups. Let $\Re $ be the set of all subgroups $E$ of  $A$ such that  $B_1  \leqslant E$ and $A/E$ is a quasicyclic $p$ -group.  By lemma 2.2, the cardinality of the set $\Re $ is continuum. Let $\aleph $ be the set of all $ji$-modules of  $A$ and let $\Im  = \{ (K + B_1 )/B_1 |K \in \aleph \} $  then it follows from lemma 2.1(ii)  that the set $\Im $ is countable. If for some $K \in \aleph $ the quotient group $(K + B_1 )/B_1 $ is finite then $\left| {K_1 /(K \cap B_1 )} \right| = n$ for some integer $n$ and hence $Kn \leqslant B_1 $. But it is impossible because $B_1 $ does not contain nonzero submodules. Thus, the quotient group $(K + B_1 )/B_1 $ is infinite for any subgroup $K \in \aleph $ and hence for any subgroup $K \in \aleph $ either $A = K + B_1 $ or $A/(K + B_1 )$ is a quasicyclic $p$-group. Since the set of all subgroups of  a quasicyclic $p$-group is countable, it implies that for any $K \in \aleph $ the set $\Re _K $ of all subgroups $U \in \Re $ such that $K + B_1  \leqslant U$ is countable. Therefore, as the set $\aleph $ is countable, the set of all subgroups $U \in \Re $ which for some $K \in \aleph $ contain the subgroup $K + B_1 $ is countable. As the cardinality of the set $\Re $ is continuum, it implies that there is a subgroup $B \in \Re $ such that $A/B$ is a quasicyclic $p$-subgroup  and $B$ does not contain subgroups from the set $\aleph $. By lemma 2.1(i)  any nonzero submodule of  $A$ contains a submodule  from $\aleph $ and hence the subgroup $B$ does not contain nonzero submodules.
\end{proof}

	A subgroup $H$ of an abelian group $A$ is said to be dense if the quotient group $A/H$ is torsion.

\begin{proposition}
     Let $A =  \oplus _{i = 1}^n A_i $, where $A_i $ are $R$-$ji$-modules, and let $p$ be a prime number. Suppose that $Sp(A_i ) \not\subset \left\{ p \right\}$ for all $ i = \overline {1,n} $  then the group $A$ has a dense subgroup $H$ such that $A/H$ is locally cyclic $p'$-group  and $H$ does not contain nonzero submodules.
    \end{proposition}
\begin{proof}
    As $Sp(A_i ) \not\subset \left\{ p \right\}$ for all $i = \overline {1,n}
    $, there is a prime number $p_1  \ne p$ such that $p_1  \in Sp(A_i )$ for some  $i$. Let $B_1 $ be the direct sum of all $A_i $ such that $p_1  \in Sp(A_i )$. If $B_1  \ne A$ then, as $Sp(A_i ) \not\subset \left\{ p \right\}$ for all $i = \overline {1,n} $, there is a prime number $p_2  \ne p_1  \ne p$ such that $p_2  \in Sp(A_i )$ for some  $i$ and $A_i  \not\subset B_1 $. Let $B_2 $ be the direct sum of all $A_i $ such that $p_2  \in Sp(A_i )$ and $A_i  \not\subset B_1 $. Continuing this process we obtain a direct decomposition $A =  \oplus _{i = 1}^m B_i $ of the module $A$, where $B_i $ is a direct sum of $ji$-modules each of whose spectrum contains prime number $p_i  \ne p$ and $p_i  \notin Sp(B_j )$ if $i < j$. By proposition 2.1, each submodule $B_i $ contains a subgroup $H_i $ such that $B_i /H_i $ is a a quasicyclic $p_i $-group and $H_i $ does not contain nonzero submodules.

     Put $H =  \oplus _{i = 1}^m H_i $, Evidently, $H$ is a dense subgroup of $A$. As all the prime numbers $p_i  \ne p$ are different, it is easy to note that $A/H$ is locally cyclic $p'$-group. Show that $H$ does not contain nonzero submodules. Suppose that $H$ contains a nonzero submodule $D$. Any nonzero element $d \in D$ may be unequally written in the form $d = \sum\nolimits_{i = 1}^m {d_i } $, where $d_i  \in H_i $. Then we can chose a nonzero element $d \in D$ with the minimal number $k$ of nonzero summands $d_i  \in H_i $ in the representation $d = \sum\nolimits_{i = 1}^m {d_i } $ of the element $d$. So, $d = \sum\nolimits_{j = 1}^k {d_{i_j } } $, where $0 \ne d_{i_j }  \in H_{i_j } $, and it follows from the minimality of $k$ that for any $\alpha  \in R$ we have $d\alpha  = 0$ if and only if $d_{i_j } \alpha  = 0$ for each $d_{i_j } $. It easily implies that the mapping $\varphi _{i_j } :dR \to d_{i_j } R$ given by $\varphi _{i_j } :d\alpha  \mapsto d_{i_j } \alpha $ is a $R$-module isomorphism for each $d_{i_j }$. Therefore, all modules $d_{i_j } R$ are isomorphic but this is impossible, because by the construction of submodules $B_i $, $p_{i_1 }  \in Sp(d_{i_1 } R)$ and $p_{i_1 }  \notin Sp(d_{i_j } R)$ for any $i_j  > i_1 $ .
\end{proof}

     \section{Essential Normal Subgroups of Soluble Groups of Finite Rank}

     Let $G$ be an infinite group, we say that an infinite normal subgroup $A$ of the group $G$ is $G$-just-infinite (or $G$-ji-subgroup for shortness)  if $\left| {A:B} \right| < \infty $ for any proper $G$-invariant subgroup $B$ of  $A$.

     \begin{lemma}     If a soluble group $G$  of finite rank has a nontrivial torsion-free normal subgroup $N$ then  $N$ has a torsion-free minimax abelian $G$-just-infinite subgroup.
     \end{lemma}
     \begin{proof}   If the group $G$ has a nontrivial torsion-free normal subgroup then it has a nontrivial abelian torsion-free normal subgroup $A$. We can consider $A$ as a $\mathbb{Z}G$-module, where $G$ acts on $A$ by conjugations, and there is no harm in assuming that $C_G (A) = 1$. Let $B = A \otimes _\mathbb{Z} \mathbb{Q}$ be the divisible hull of $A$. It follows from the results of \cite{2} that the group $G$ has a finite series of normal subgroups $1 \leqslant N \leqslant H \leqslant G$, where the quotient group $H/N$ is finitely generated abelian, the quotient group $G/H$ is finite, and the group $B$ has a nontrivial G-invariant subgroup $C$ centralized by $N$. Then replacing $A$ by $A \cap C$ we can assume that $G/C_{_G } (A)$ is finitely generated abelian-by-finite. It follows from lemma 5.1 of \cite{5} that $a\mathbb{Z}G$ is a minimax group for any $0 \ne a \in A$ and replacing $A$ by $a\mathbb{Z}G$ we can assume that $A$ is a minimax group. Since, by \cite{11}, for any descending chain $\left\{ {A_i } \right\}$ of subgroups of $A$ there is an integer $n$ such that $\left| {A_i :A_{i + 1} } \right| < \infty $ for all $i \geqslant n$, it is easy to show that  there is a nontrivial $G$-just-infinite subgroup $D \leqslant A$.
     \end{proof}

     \begin{lemma} Let $G$ be a soluble group of finite rank which has a nontrivial torsion-free normal subgroup. Then the group $G$ has a torsion-free minimax abelian normal subgroup $jiSoc(G) \ne 1$ such that $jiSoc(G)$ is a direct product of finitely many of torsion-free abelian minimax $G$-just-infinite subgroups and $jiSoc(G) \cap N \ne 1$ for any nontrivial torsion-free normal subgroup $N$ of the group $G$.
     \end{lemma}
     \begin{proof}
	We construct the subgroup $jiSoc(G) = S$ by induction. By lemma 2.1, the group $G$ has a torsion-free minimax abelian $G$-just-infinite subgroup $A$ and we put $S_1  = A$. If there is a nontrivial torsion-free normal subgroup $N$ of the group $G$ such that $S_1  \cap N = 1$ then it follows from lemma 3.1 that there is a torsion-free minimax abelian $G$-just-infinite subgroup $1 \ne A_1 \leqslant N$ such that $S_1  \cap A_1 = 1$ and we put $S_2  = S_1  \times A_1$. Continuing this process we should note that it is terminated because the group $G$ has finite rank and hence the subgroup $jiSoc(G)$ does exist.
\end{proof}

 If the group $G$ has no nontrivial  torsion-free normal subgroup then we put $jiSoc(G) = 1$
            , it follows from the above lemma that $jiSoc(G) \ne 1$ if and only if the group $G$ has a nontrivial torsion-free normal subgroup. Certainly, the subgroup $jiSoc(G)$ is not defined uniquely.

      \begin{proposition}
     Let $G$ be a soluble group of finite rank and let $p$ be a prime number . Suppose that $Sp(B) \not\subset \{ p\} $ for any nontrivial abelian torsion-free normal subgroup $B$ of  $G$. If $jiSoc(G) \ne 1$ then $jiSoc(G)$ has a dense subgroup $ H$ such that $jiSoc(G)/H$ is a locally cyclic torsion $p'$-group and $H$ does not contains nontrivial $G$-invariant subgroups.
    \end{proposition}
    \begin{proof}
     As $Sp(B) \not\subset \{ p\} $ for any nontrivial abelian torsion-free normal subgroup $B$ of  $G$, it follows from lemma 2.1 that $Sp(B) \not\subset \{ p\} $ for any abelian torsion-free $G$-just-infinite subgroup $B$ of  $G$. We can consider $jiSoc(G)$ as a $\mathbb{Z}G$-module, where the group $G$ acts on $jiSoc(G)$ by conjugations. Then $jiSoc(G) =  \oplus _{i = 1}^n A_i $, where $A_i $ are $\mathbb{Z}G$-$j.i.$-modules. Therefore, $jiSoc(G)$  and $p$ meet the conditions of proposition 2.2 and the assertion follows.
     \end{proof}

     \begin{proposition} Let $G$ be a soluble group of finite. Then $jiSoc(G) \times Soc(G)$ is an essential subgroup of $G$ and for any normal subgroup $1 \ne N \leqslant G$ either $jiSoc(G) \cap N \ne 1$ or $Soc(G) \cap N \ne 1$.
         \end{proposition}
         \begin{proof}
    The subgroup $N \leqslant G$ contains a nontrivial abelian $G$-invariant subgroup and hence we can assume that the subgroup $N$ is abelian. If the torsion subgroup $T$ of $N$ is trivial than it follows from lemma 3.2 that  $jiSoc(G) \cap N \ne 1$. If $T \ne 1$ then , as $r(G) < \infty $, $T$ contains a nontrivial finite $G$-invariant subgroup and hence $Soc(G) \cap N \ne 1$.
    \end{proof}

    \section{Essential Subgroups and Representations of Groups}
    \begin{proposition}
         Let  $k$ be a field and $H =  \times _{i = 1}^n H_i $ be a group. Let $\varphi _i $ be an irreducible representation of the group $H_i $ over the field $k$, where $i = \overline {1,n} $. Then there is an irreducible representation $\varphi $ of the group $H$ over the field $k$ such that $Ker\varphi  \cap H_i  = Ker\varphi _i $, where $i = \overline {1,n} $.
         \end{proposition}
         \begin{proof}
    Let $M_i $ be a $kH_i $-module of the representation $\varphi _i $ then, as the module $M_i $ is irreducible, there is a generator $a_i $ of the module $M_i $, where $i = \overline {1,n} $. We can consider $M_i $ as a $kH$-module, where $H_j $  acts on $M_i $  trivially for any $j \ne i$.

    Put $M = M_1  \otimes _F M_2  \otimes _F ... \otimes _F M_n $, where $H$ acts on $M$ as the following : $(m_1  \otimes m_2  \otimes ... \otimes m_n )g = m_1 g \otimes m_2 g \otimes ... \otimes m_n g$ ( see \cite{6}, Chap. XVIII, \S 2 ). Then it is not difficult to show that the $kH$-module $M$ is generated by $a = a_1  \otimes a_2  \otimes ... \otimes a_n $ and $Ann_{FH_i } (a) = Ann_{FH_i } (a_i )$ for any $i = \overline {1,n} $. So, $akH_i  \simeq M_i $ and we can assume that $M_i  = akH_i $, where $i = \overline {1,n} $.

    Let $L$ be a maximal submodule of $M$. Put $W = M/L$ and let $\varphi $ be an irreducible representation of the group $H$ induced by action of $H$ on the $kH$-module $W$. Since the $kH_i $-module $M_i  = akH_i $ is irreducible , $M_i  \cap L = 0$ and hence we can assume that $M_i  \leqslant W$. Therefore, $Ker\varphi _i  = C_{H_i } (M_i ) \geqslant C_{H_i } (W) = Ker\varphi  \cap H_i $. On the other hand, it follows from the definition of action of $H_i $ on $M$ that $Ker\varphi _i  = C_{H_i } (M_i ) \leqslant C_{H_i } (M) \leqslant C_{H_i } (W) = Ker\varphi  \cap H_i $ and we can conclude that $(Ker\varphi  \cap H_i ) = Ker\varphi _i $, where $i = \overline {1,n} $.
    \end{proof}

   \begin{proposition}Let $G$ be a group with a normal  subgroup  $H =  \times _{i = 1}^n H_i $, where all $H_i $ are normal subgroups of $G$, such that $N \cap H_i  \ne 1$   for any normal subgroup $1 \ne N \leqslant G$ and for some subgroup $H_i $. Let  $k$ be a field and suppose that each subgroup $H_i $ has an irreducible representation $\varphi _i $ over the field $k$  such that $Ker\varphi _i $ does not contain nontrivial $G$-invariant subgroups. Then the group $G$ has an irreducible faithful representation over the field $k$.
    \end{proposition}

    \begin{proof}
     By proposition 4.1, there is an irreducible representation $\varphi $ of the subgroup $H$ over the field $k$ such that $Ker\varphi  \cap H_i  = Ker\varphi _i $, where $i = \overline {1,n} $. Let $M$ be a module of the representation $\varphi $ then $M \cong kH/J$, where $J$ is a maximal ideal of $kH$. Let $I$ be a maximal ideal of $kG$ such that $J \leqslant I$ then $I \cap kH = J$ and $W = kG/I$ is an irreducible $kG$-module. Let $\phi $ be a representation of the group $G$ over the field $k$ induced by action of $G$ on $W$. As $I \cap kH = J$, we can assume that $M \leqslant W$ and hence $Ker\phi  \cap H \leqslant Ker\varphi $. Then, as $Ker\varphi  \cap H_i  = Ker\varphi _i $, we can conclude that $Ker\phi  \cap H_i  \leqslant Ker\varphi _i $ for any $i = \overline {1,n} $. If $Ker\phi  \ne 1$ then $Ker\phi  \cap H_i  \ne 1$ for some $i$ but it is impossible because $Ker\phi  \cap H_i  \leqslant Ker\varphi _i $ and  $Ker\varphi _i $ does not contain nontrivial $G$-invariant subgroups.
    \end{proof}

     \begin{lemma} Let $G$ be a group and let $k$ be a  field. Let $M$ be a simple $kG$-module such that  $C_G (M) = 1$. Then :\par
     (i) $M$ contains a simple  $kH$-module for any subgroup $H \leqslant G$ of finite index in $G$;\par
      (ii) $C_M (N) = 0$ for any nontrivial normal subgroup $N$ of $G$;\par
    (iii) ) if $A$ is a central subgroup of $G$ then $Ann_{kA} (M)$ is a prime ideal of $kA$;\par
(iv) if $A$ is an elementary abelian normal $p$-subgroup of $G$ such that $\left| {G:C_G (A)} \right| < \infty $ then $p \ne chark$ and $A$ contains a subgroup $H$ such that $A/H$ is a cyclic group and $H$ contains no nontrivial  $G$-invariant subgroups .
\end{lemma}
\begin{proof}
     (i) Let $K = \bigcap\nolimits_{g \in G} {H^g } $  then $K$ is a normal subgroup of finite index in $G$ such that $K \leqslant H$. Let $V$ be a maximal $kK$-submodule of $M$. As $\left| {G:K} \right| < \infty $ , it is easy to note that the set $\left\{ {Vg|g \in G} \right\}$ of $kK$-submodules of $M$ is finite. Since $ \cap _{g \in G} Vg \ne M$ is a  $kG$-submodule of $M$, we see that $ \cap _{g \in G} Vg = 0$. Then, by Remak theorem, $M \leqslant  \oplus _{g \in G} M/Vg$. As the set $\left\{ {Vg|g \in G} \right\}$ is finite, so is the set $\left\{ {M/Vg|g \in G} \right\}$ of simple $kK$-modules $M/Vg$ and hence, as $M \leqslant  \oplus _{g \in G} M/Vg$ , we can conclude that $M$ is an arthenian $kK$-module. Therefore, as $K \leqslant H$, $M$ is an arthenian $kH$-module and hence $M$  contains a simple  $kH$-module.

	(ii) Suppose that $C_M (N) = W \ne 0$. Since $wgh = w(ghg^{ - 1} )g = wg$ for any $w \in W$, $h \in N$ and $g \in G$, we can conclude that $W$ is a nonzero $kG$-submodule of the simple module $M$. Therefore, $C_M (N) = W = M$
and hence $1 \ne N \leqslant C_G (M)$ but this is impossible because $C_G (M) = 1$.\par

	(iii) As $A$ is a central subgroup of $G$, it is not difficult to show that $Ma$ is a submodule of $M$ for any element $a \in kA$. Then, as the module $M
$ is simple, it is easy to note that either $Ma = 0$ or $Ma = M$ for any element $a \in kA$. It implies that if $Mab = 0$ for some $a,b \in kA$ then either $a \in Ann_{kA} (M)$ or $b \in Ann_{kA} (M)$. Thus $Ann_{kA} (M) $ is a prime ideal of $kA$.

(iv) Put $C = C_G (A)$ then it follows from (i) that $M$ contains a simple $kC$-module $W$. It follows from (iii) that $Ann_{kA} (W)$ is a prime ideal of $kA$. If $p = char k$ then $kA$ has an unique prime fundamental ideal $I = \left\langle {1 - g|g \in A} \right\rangle $. But then $0 \ne W \leqslant C_M (A)$ that contradicts (ii). Thus, $p \ne chark$.

	Since $Ann_{kA} (W)$ is a prime ideal of $kA$, the quotient group $A/C_A (W)$ is a subgroup of the multiplicative group of the field of fractions of the domain $kA/Ann_{kA} (W)$. Then, as  $A$ is an elementary abelian $p$-group, it follows from theorem 127.3 of \cite{3} that the quotient group $A/C_A (W)$ is cyclic. If $C_A (W)$  contains a $G$-invariant nontrivial subgroup $K$ then $0 \ne W \leqslant C_M (K)$ that contradicts (ii). Thus, we can put $H = C_A (W)$.
\end{proof}
   \begin{proposition}
        Let $G$ be a group such that all minimal abelian normal subgroups of $G$ are finite and let $k$ be a  field. If the group $G$ has a faithful irreducible representation $\varphi $   over the field $k$ then $chark \notin \pi (abSoc(G))$ and one of the following equivalent conditions holds:\par
    (i) the abelian socle $abSoc(G)$ contains a subgroup $H$ such that the quotient group $abSoc(G)/H$ is locally cyclic and $H$ contains no nontrivial $G$-invariant subgroups;\par
    (ii) the abelian socle $abSoc(G)$ is a locally cyclic $\mathbb{Z}G$-module, where the group $G$ acts on $abSoc(G)$ by conjugations.
    \end{proposition}
\begin{proof}
     	Let $M$ be a simple $kG$-module of the representation $\varphi $ then $C_G (M) = 1$. Put $chark = q$ if $chark \in \pi (abSoc(G))$ then the group $G$ has a finite nontrivial normal abelian $q$-subgroup $A$. As $\left| A \right| < \infty $, we can conclude that $\left| {G:C_G (A)} \right| < \infty $ but it contradicts lemma 4.1(iv). Thus,  $chark \notin \pi (abSoc(G))$.\par
     (i) We can consider $A = abSoc(G)$ as a semi-simple $\mathbb{Z}G$-module all of whose simple submodules are finite, where the group $G$ acts on $abSoc(G)$ by conjugations. Let $B = \mathop  \oplus \limits_{i \in I} B_i $ be an isotype component of $A$ (see \cite{1}, Chap. VIII, §3), where $B_i $ are isomorphic simple $\mathbb{Z}G$-modules, then $B$ is an elementary abelian $p$-group. Since the modules $B_i $ are isomorphic and $B = \mathop  \oplus \limits_{i \in I} B_i $, it is easy to show that $C_G (B_i ) = C_G (B)$ for all $i \in I$ and, as all submodules $B_i $ are finite, we can conclude that $|G:C_G (B)| < \infty $.

    Let $A_p  = \mathop  \oplus \limits_{i \in I} A_i $ be a Sylow $p$-component of $A$, where $A_i $ are isotype components of $A_p $. As it was shown above, $|G:C_G (A_i )| < \infty $ and hence, by lemma 4.1(iv), each component $A_i $ has a subgroup $H_i $ which contains no nontrivial $G$-invariant subgroups and such that $|A_i /H_i | = p$. Then, as $A_i $ is an elementary abelian $p$-group, there is an element $a_i  \in A_i $, such that $A_i  = H_i  \oplus \left\langle {a_i } \right\rangle $. We fix an index $i_1  \in I$ and put $H_p  = \left\langle {\{ a_i  - a_{i_1} |i_1  \ne i\}  \cup \{ H_i |i \in I\} } \right\rangle $. Then it is not difficult to note that $|A_p /H_p | = p$ and $H_p  \cap A_i  = H_i $.

    Suppose that $H_p $ contains a nontrivial $\mathbb{Z}G$-submodule then $H_p $ contains a simple submodule $T$. By the definition of isotype components, $T \leqslant A_i $ for some $i \in I$ and it implies that $T \leqslant H_p  \cap A_i $. However, $H_p  \cap A_i  = H_i $ and hence $T \leqslant H_i $ but this is impossible because $H_i $ contains no nonzero $\mathbb{Z}G$-submodules. Thus, each Sylow $p$-component  $A_p$   of $A$ has a subgroup $H_p $ which contains no nonzero $\mathbb{Z}G$-submodules and such that the quotient group $A_p /H_p $ is cyclic. Put $H = \mathop  \oplus \limits_{p \in \pi (A)} H_p $ then it is not difficult to show that the quotient group $A/H$ is locally cyclic and $H$ contains no nontrivial $G$-invariant subgroups.

	(ii) It follows from   lemma 7 of \cite{8} that  conditions (i) and (ii) are equivalent.
\end{proof}

	\section{Some Necessary and Sufficient Conditions of Existence of Faithful Irreducible Representations of Soluble Groups of Finite Rank}

     \begin{lemma} Let $A$ be a torsion-free abelian group of finite rank and let $k$ be a field.

(i) if the field $k$ is not locally finite then the group $A$ has a faithful irreducible representations over the field $k$;

(ii) if the field $k$ is  locally finite of characteristic $p$ and the group $A$ has a dense subgroup $H$ such that $A/H$ is a locally cyclic $p'$-group then the group $A$ has an irreducible representations $\varphi $   over the field $k$  such that  $Ker\varphi  = H$.
 \end{lemma}
\begin{proof}
Let $k^*$ be the multiplicative group of the field $k$, let $\hat k$  be the algebraic closure of the field $k$ and let $\hat k^*$ be the multiplicative group  of the field $\hat k$. At first, we note that any group homomorphism $\varphi :A \to \hat k^*$ may be continued to a ring homomorphism $\varphi :kA \to \hat k$ given by $\varphi :\sum\nolimits_{a \in A} {k_a a}  \mapsto \sum\nolimits_{a \in A} {k_a \varphi (a)} $, where $k_a  \in k$. Besides, $\varphi (kA)$ is a subfield of $\hat k$ because $k \leqslant \varphi (kA)$ and for any element $a \in \hat k$ the subring of $\hat k$ generated by $k$ and $a$ is a subfield of $\hat k$ (see \cite{6}, Chap. VI, Proposition 3). Therefore, $\varphi (kA)$ is an irreducible $kA$-module, where action of $\alpha  \in kA$ on $\varphi (kA)$ is defined by multiplication of elements of $\varphi (kA)$ by $\varphi (\alpha )$, and $\varphi $ is an irreducible representation of the group $A$ over the field $k$.

     (i) Since the field $k$ is not locally finite, either $\mathbb{Q} \leqslant k$  or $f(t) \leqslant k$, where $f(t)$ is the field of fractions of the group algebra of infinite cyclic group $ \left\langle t \right\rangle $ over a finite field $f$. It is well known that the multiplicative groups $\mathbb{Q}^*$ and $f(t)^* $ are not torsion. Then it follows from theorem 127.3 of \cite{3}  that $\hat k^*$ has a torsion-free divisible subgroup of infinite rank. Therefore, it is not difficult to note that there is a group monomorphism $\varphi :A \to \hat k^*$ and $\varphi $ is a faithful irreducible representation of the group $A$ over the field $k$.

	(ii) It follows from theorem 127.3 of \cite{3} that $\hat k^*$ is a direct product of quasicyclic $q$-groups, where $q$ runs through the set of all prime numbers except $p$. Then it is not difficult to note that there is a group homomorphism $\varphi :A \to \hat k^*$ such that $Ker\varphi  = H$. Therefore, $\varphi $ is an irreducible representation of the group $A$ over the field $k$ such that $Ker\varphi  = H$.
\end{proof}
     \begin{theorem} Let $G$ be a soluble group of finite rank  and let $k$ be a  field. If the group $ G$ has a faithful irreducible representation over the field $k$ then  $Soc(G)$ is a locally cyclic $\mathbb{Z}G$-module, where the group $G$ acts on $Soc(G)$ by conjugations, and $chark \notin \pi (Soc(G))$.
\end{theorem}
   \begin{proof}
   Since the group $G$ has finite rank, all minimal abelian normal subgroups of $G$ are finite and the assertion follows from Proposition 4.3.
\end{proof}

	\begin{theorem} Let $G$ be a soluble group of finite rank and let $k$ be a not locally finite field. The group $G$ has a faithful irreducible representation over the field $k$ if and only if $ Soc(G)$ is a locally cyclic $\mathbb{Z}G$-module, where the group $G$ acts on $Soc(G)$ by conjugations, and $chark \notin \pi (Soc(G))$.
\end{theorem}

	\begin{proof}
If the group $G$ has a faithful irreducible representation over the field $k$ then, by theorem 5.1, $Soc(G)$ is a locally cyclic $\mathbb{Z}G$-module, where the group $G$ acts on $abSoc(G)$ by conjugations, and $chark \notin \pi (Soc(G))$.

	Suppose now that $Soc(G)$ meets the conditions of theorem. Then it follows from proposition 4.3 that $Soc(G)$ has a subgroup $H$ such that $Soc(G)/H$ is a locally cyclic $p'$-group, where $chark = p$, and $H$ contains no nontrivial $G$- invariant subgroups. By lemma 5.1(i), $jiSoc(G)$ has a faithful irreducible representation over $k$ and, by lemma 5.1(ii), $Soc(G)$ has an irreducible representation $\varphi $ over $k$ such that $Ker\varphi  = H$. By Proposition 3.2, for any nontrivial normal subgroup $N$ of $ G$ either $jiSoc(G) \cap N \ne 1$ or $Soc(G) \cap N \ne 1$. Then it follows from proposition 4.2 that $G$ has a faithful  irreducible representation over $k$.
\end{proof}

     We also obtained a criterion of existence of faithful irreducible representations of soluble groups of finite rank over a locally finite field under some additional conditions. .

     \begin{theorem} Let $G$ be a soluble group of finite rank and let $k$ be a locally finite field. Suppose that $Sp(B) \not\subset \{ chark\} $ for any nontrivial  abelian torsion-free normal subgroup $B$ of $G$.  Then the group $G$ has a faithful irreducible representation over the field $k$   if and only if  $Soc(G)$ is a locally cyclic $\mathbb{Z}G$-module, where the group $G$ acts on $Soc(G)$ by conjugations, and  $chark \notin \pi (Soc(G))$.
\end{theorem}
\begin{proof}
 If the group $G$ has a faithful irreducible representation over the field $k$ then, by theorem 5.1, $ Soc(G)$ is a locally cyclic $\mathbb{Z}G$ -module, where the group $G$ acts on $Soc(G)$ by conjugations, and $char k \notin \pi (Soc(G))$.

	Suppose now that $Soc(G)$ meets the conditions of the theorem. Then it follows from proposition 4.3 that $Soc(G)$ has a subgroup $H_1 $ such that $Soc(G)/H_1 $ is a locally cyclic $p'$-group, where $chark = p$, and $H_1 $ contains no nontrivial $G$- invariant subgroups. By proposition 3.1, $jiSoc(G)$ has a subgroup $H_2 $ such that $jiSoc(G)/H_2 $ is a locally cyclic $p'$-group, where  $chark = p$, and $H_2 $ contains no nontrivial $G$- invariant subgroups. By lemma 5.1, $Soc(G)$ has a an irreducible representation $\varphi _1 $ over $k$ such that $Ker\varphi_1  = H_1 $ and $jiSoc(G)$ has an irreducible representation $\varphi _2 $ over $k$ such that $Ker\varphi_2  = H_2 $. By Proposition 3.2, for any nontrivial normal subgroup $N$ of $G$ either $jiSoc(G) \cap N \ne 1$ or $Soc(G) \cap N \ne 1$. Then it follows from proposition 4.2 that $G$ has a faithful  irreducible representation over $k$.
\end{proof}


\begin{thebibliography}{0}





\bibitem{1} N. Bourbaki, {\small\it Elements of mathematics. Algebra II }(Springer, 2003).

\bibitem{2} V. S. Charin, On groups of automorphisms of nilpotent groups, {\small\it Ukrain. Mat. Zh.} {\small\bf 6}  (1954) 295–-304. (in Russian)

\bibitem{3} L. Fuchs, {\small\it Infinite abelian groups}, Vol.~ II (Acadimic Press, New York, London, 1973)

\bibitem{4} W. Gaschutz, Endliche Gruppen mittreuen absolutirreduziblen, {\small\it Darstellungen, Math. Nachr.} {\small\bf 12}   (1954) 253--255.

\bibitem{5} Ph. Hall, On the finiteness of certain soluble groups,{\small\it  Proc. London Math. Soc. } {\small\bf 3}  (1959) 595--622.

\bibitem{6} S. Lang, {\small\it Algebra} (Addison-Wesley Publishing Company, Mass, 1965).

\bibitem{7} J. E. Roseblade, Groups rings of polycyclic groups, {\small\it J. Pure and Appl. Algebra.} {\small\bf 3}  (1973) 307--328.

\bibitem{8} A.V. Tushev, Irreducible representations of locally polycyclic groups over an absolute field, {\small\it Ukrainian Math. J.} {\small\bf 42}   (1990) 1233--1238.

\bibitem{9}  A. V. Tushev,  On exact irreducible representations of locally normal groups ,{\small\it Ukrainian Math. J.} {\small\bf 45}
   (1993) 1900--1906.

\bibitem{10} B.A.F. Wehrfritz, Groups whose irreducible representations have finite degree , {\small\it Math. Proc. Camb. Phil. Soc.} {\small\bf 90}   (1981) 411--421.

\bibitem{11} D.I. Zaitsev, Groups which satisfy the weak minimal condition , {\small\it Ukrain. Mat. Zh.} {\small\bf 20}   (1968) 472--482. (in Russian).

\end{thebibliography}
\end{document}